%% file: VJ26main.tex
\documentclass[11pt,a4paper]{article}
\usepackage{macros}

\title{Norm Bounds for Sparse Random Tensors and\\ Spectral Gap of Random Hypergraphs}

\newcommand{\email}[1]{\href{mailto:#1}{\texttt{#1}}}

\author{
  Kevin Lucca%
  \thanks{ETH Zürich. \email{kevin.lucca@ifor.math.ethz.ch}}
  \and
  Lucas Pesenti%
  \thanks{ETH Zürich. \email{lpesenti@ethz.ch}.}
}

\begin{document}

    \maketitle

    \begin{abstract}
        Friedman and Wigderson (1995) introduced a notion of second
        eigenvalue for hypergraphs that generalizes the second eigenvalue of the
        adjacency matrix of a graph. We show that $r$-uniform
        \Erdos-\Renyi hypergraphs on $n$ vertices exhibit a spectral gap
        as soon as their expected number of hyperedges $m$ satisfies $m \gg n^{r/2}$. Prior work identified this scale only up to logarithmic factors; removing these factors is the main technical challenge.
        
        Our proof overcomes this obstacle through an explicit decomposition of an associated selector process,
        inspired by a generic decomposition
        theorem of
        Talagrand (2021).
        As a consequence of our techniques, we obtain improved norm bounds for sparse random tensors with independent entries. Finally, under a mild moment equivalence assumption, we extend to tensors a seminal result of Seginer (2000) for random matrices with i.i.d.\ entries.
    \end{abstract}

    \input{IntroVJ26}

\input{BernsteinPartVJ26}

\input{WitnessConstructionVJ26}

\input{PositivePartVJ26}
    \input{ExtensionsVJ26}

    \phantomsection
    \addcontentsline{toc}{section}{References}
	\bibliographystyle{alpha}
    { \bibliography{main}}

    \appendix
    \input{AuxillaryLemmasVJ26}

\end{document}

%% file: IntroVJ26.tex
\section{Introduction}
\label{sec:intro}

Can we develop a spectral theory for hypergraphs that matches the power of the
well-developed spectral theory of graphs? This long-standing question motivates
the development of new tools for studying matrices that extend to tensors.

In this line of work, Friedman and Wigderson~\cite{friedmanWigderson} introduced a notion of second
eigenvalue for hypergraphs that generalizes the second eigenvalue of
the adjacency matrix of a graph.
They also asked for the value of this second eigenvalue for {\em random} hypergraphs.

\begin{definition}[\Erdos-\Renyi model]\label{def:erdos}
    Let $n, r\in \N$ and $p=p(r,n)\in [0,1]$. Consider the distribution over
	$r$-uniform hypergraphs on the vertex set $[n] \defeq \{1, \ldots, n\}$
    obtained by independently including each subset 
    of vertices $e\subseteq [n]$
    of size $|e|=r$ as a hyperedge with
    probability $p$. 
\end{definition}

\noindent We represent a hypergraph
by its adjacency tensor $T\in (\R^n)^{\otimes r}$, defined by $T_{i_1,\ldots,i_r}=1$
if $\{i_1, \ldots, i_r\}$
is a hyperedge, and $T_{i_1, \ldots, i_r}=0$
otherwise.
Let $\calH_r(n,p)$ 
denote the distribution of the 
adjacency tensors
of \Erdos-\Renyi hypergraphs.

\begin{definition}[Injective tensor norm]
    The injective norm of $T\in (\R^n)^{\otimes r}$ is
    \begin{equation*}
        \normin{T} \defeq \max_{X\in \mathcal X} \Abs{\Iprod{T, X}}\,,
    \end{equation*}
    where $\mathcal X \defeq \{x_1\otimes \ldots \otimes x_r : \|x_1\|_2= \ldots= \|x_r\|_2 = 1\}$ is the set
    of unit rank-1 tensors.
\end{definition}

When $T$ is
invariant under permutations of its indices,
the injective norm can equivalently be written as
the coupled
maximization problem over $\{x^{\otimes r} : \|x\|_2 = 1\}$~\cite{ancestralPaper}.

Following the terminology of~\cite{friedmanWigderson}, when $T\sim \calH_r(n,p)$, we call
$\normin{T-\E  T}$ the {\em second eigenvalue}\footnote{This terminology is not entirely standard, as one usually subtracts the top eigenvector contribution from $T$ to define the second eigenvalue.} of $T$,
and $\normin{T}-\normin{T-\E  T}$ the {\em spectral gap} of
$T$. More explicitly, the second eigenvalue of an \Erdos-\Renyi hypergraph is:
\[
    \normin{T - \E T} = \max_{\|x_1\|_2 = \ldots = \|x_r\|_2 = 1} \sum_{\substack{i_1,\ldots,i_r=1\\\text{distinct}}}^n (T_{i_1,\ldots,i_r} - p) x_{1,i_1}\ldots x_{r,i_r}\,.
\]

In the graph case ($r=2$), $\normin{ T}$ and $\normin{T-\E  T}$ are,
respectively, the spectral norm of the uncentered and centered adjacency
matrices of an \Erdos-\Renyi{} random graph. 
These quantities have been
central in the development of spectral graph theory. In sparse random graphs,
the emergence of a spectral gap coincides with the existence
of efficient algorithms certifying (nearly) tight bounds on large cliques and independent sets; see Section~\ref{sec:refutation}. The emergence of a spectral gap also marks the region in which the expander mixing lemma holds; see Example~\ref{ex:mixingLemma}.

\subsection{Emergence of a spectral gap in \Erdos-\Renyi hypergraphs}

Motivated by this analogy, we determine
the sparsity scale at which
the spectral gap of \Erdos-\Renyi
hypergraphs diverges.
The upper bound 
\[
    \normin{T}-\normin{T-\E T}\le \normin{\E T}\le pn^{r/2}
\] 
shows
that $p\gg n^{-r/2}$ is a necessary condition
for the spectral gap to diverge.
Our first main result shows that, for $r\ge 3$, this condition is
also sufficient.

\begin{theorem}[Spectral gap of \Erdos-\Renyi hypergraphs; consequence of Theorem~\ref{thm:mainTensorProof}]
    \label{thm:mainTensor}
    Let $r\ge 3$ be fixed, and let $p=p(n)$ be such that $pn^{r/2}\to\infty$ as $n\to\infty$. Then,
    with high probability over $T\sim \calH_r(n,p)$,
    \[
        \normin{T} - \normin{T - \E T} = (1-o(1)) pn^{r/2}\qquad \text{as $n\to\infty$}\,.
    \]
\end{theorem}

In the graph case $r=2$, the matrix Bernstein inequality
shows that $p\gg \log n/n$ is sufficient for
the spectral gap to diverge.
This is also the regime of $p$
studied in the original paper of Friedman
and Wigderson on random tensors~\cite{friedmanWigderson}.
Theorem~\ref{thm:mainTensor} shows that, unlike in graphs, this regime
lies far beyond the point $p \gg n^{-r/2}$ at which a spectral gap first emerges.

Theorem~\ref{thm:mainTensor} overcomes
a barrier present in several prior works. Using tensor flattening,
Zhou and Zhu~\cite[Theorem 2.3]{sparseTensorZhou} obtain an upper bound on the
second eigenvalue throughout the sparsity regime of
Theorem~\ref{thm:mainTensor}. 
However, their bound always carries at least a
polylogarithmic factor in $n$, and therefore cannot establish the phase
transition. As we explain in Section~\ref{sec:overview}, this loss is
inherent to the flattening approach.
Similar bounds arise in the analyses of
algorithms for tensor completion~\cite[Theorem 2.1]{tensorCompletionJain} and multilayer community
detection~\cite[Theorem 2]{multiLayerSBM},
but they likewise incur extra
$\polylog n$ factors. Section~\ref{sec:applis} discusses consequences of our techniques for these problems.

\subsection{Motivation: refutation of random constraint satisfaction problems}
\label{sec:refutation}

Theorem~\ref{thm:mainTensor} yields a certificate that upper bounds the size of the maximum
independent set in a random hypergraph. 
For a hypergraph, a subset of vertices $S$ is an independent set if no
hyperedge is fully contained in $S$. If $S$ is an independent set of the hypergraph represented by $T \sim \calH_r(n,p)$, the
indicator vector $\bs 1_S$ satisfies
\[
    \frac{\Abs{\Iprod{ T - \E  T,  \mathbf 1_S^{\otimes r}}}} {\norm{\mathbf 1_S^{\otimes r}}_2} = \frac{p|S|(|S|-1)\ldots(|S|-r+1)} {|S|^{r/2}}\,.
\]
Hence, when $p \gg n^{-r/2}$, Theorem~\ref{thm:mainTensor} implies that such random hypergraphs typically do not contain
linear-sized independent sets, and this is
certified by the second eigenvalue of
$ T$.

Together with known reductions from the problem of
refuting random constraint satisfaction
problems to that of certifying the absence of linear-sized independent sets in random
hypergraphs~\cite{goerdtKrivilevich}, this suggests that the second eigenvalue could provide an efficient
certificate of unsatisfiability for random $r$-SAT formulas
with $O(n^{r/2})$ constraints.
For even $r$,
the existence of such certificates
already follows from a flattening
argument and Grothendieck's inequality~\cite{randomRefute}.
However, for odd $r$, such certificates were not known until the
recent work~\cite{lucaTom}, which uses a very different spectral certificate based on the
non-backtracking walk matrix and has a significantly more
involved analysis.

Our work is motivated by the question of whether the second eigenvalue of sparse random hypergraphs can be
computed efficiently. 
While there is strong evidence that, for dense random tensors ($p=1/2$), the second eigenvalue can only be approximated
within an $\Omega(n^{1/4})$ factor~\cite{tensorPCASOS,subexpLBPCA}, it is unclear whether such lower bounds
extend to the sparse setting. In fact, a simple flattening argument shows that the integrality gap of
the basic SDP relaxation at the threshold $p = n^{-r/2}$ is at most $O(\log n)$. 

Our proof technique decomposes test vectors into several
components, with the aim of clarifying which components
might be approximated efficiently. This decomposition is guided by the
theory of selector processes, as we discuss further in
Section~\ref{sec:overview}. We leave the algorithmic aspects of this
decomposition to future work. More broadly, it suggests a possible route
towards rehabilitating the ``Chernoff + union bound'' approach for the design of
efficient approximation algorithms.

\subsection{Concentration of random tensors}

On our way to prove Theorem~\ref{thm:mainTensor}, we show the
following general result on the injective norm
of sparse random tensors:

\begin{theorem}[Concentration of sparse random tensors; see Corollary~\ref{corollary:unboundedsparsetensor}]\label{thm:tensorVersion}
    Let $T\in (\R^n)^{\otimes r}$ be a random tensor with
    independent entries such that
    for some $\eps\in (0,1)$,
    $\Pr\left(T_i\neq 0\right)\le n^{-1-\eps}$
    for all $i\in [n]^r$. Then,
    \[
        \E \normin{T - \E T} \lesssim \frac {r^5} {\eps^2} \log^2 \frac r \eps \cdot \E  \max_{i\in [n]^r} \abs{T_i}\,.
    \]
\end{theorem}

Thus, in the sparse regime where each
entry is nonzero with probability $\ll 1/n$, the
injective norm is controlled, up to a polynomial factor in $r$, by the largest entry of the tensor.
Previous bounds for sparse random tensors, such as~\cite{sparseTensorZhou}, not only depend on $n$, but also depend exponentially on $r$. Improving this dependence was explicitly posed
there as a direction for future work. 
We do not expect the particular
polynomial dependence on $r$ and $\eps$ in Theorem~\ref{thm:tensorVersion} to be optimal.
We note that even for tensors with i.i.d.\ Gaussian entries, the precise asymptotic constant in the injective norm was determined only recently~\cite{2024DartoisMcKenna,bates2025balanced}.

Theorem~\ref{thm:mainTensor} 
is a special case of Theorem~\ref{thm:tensorVersion}, where $T$ has
i.i.d.\ entries up to the symmetry. A landmark result on the
norm of general random matrices with
i.i.d.\ entries is Seginer's theorem:

\begin{theorem}[{\cite[Corollary 2.2]{seginer}}]\label{thm:seginerMatrix}
    Let $A\in \R^{n\times m}$ 
    be a random matrix with
    i.i.d.\ mean-zero entries. Then,
    \[
        \E \normin{A} \asymp \E \max_{i\in [n]} \left(\sum_{j=1}^m A_{ij}^2\right)^{1/2} + \E \max_{j\in [m]} \left(\sum_{i=1}^n A_{ij}^2\right)^{1/2}\,.
    \]
\end{theorem}

\noindent For example, when 
applied to the centered adjacency matrix of \Erdos-\Renyi
random graphs, 
Theorem~\ref{thm:seginerMatrix} implies that the expected second eigenvalue is within a constant factor of the expected square root of the maximum degree.

Remarkably, Theorem~\ref{thm:seginerMatrix}
does not require any assumption
on the distribution of the entries. Recently,~\cite{latala2025operator} established a generalization of Theorem~\ref{thm:seginerMatrix} to the $\ell_p \to \ell_q$ operator norm of matrices, under a moment equivalence assumption on the distribution of the entries.
Building on Theorem~\ref{thm:tensorVersion}, we derive a generalization of
Theorem~\ref{thm:seginerMatrix}
for the injective norm of tensors, under the following moment equivalence assumption:

\begin{assumption}\label{ass:momentGrowthAssumption}
    Let $Z$ be a mean-zero random variable such that, for some $\eps\in (0,1)$ and $K\ge 1$,
    \[
        \E |Z|^{2+\eps}\le K \sigma^{2+\eps}\,,\qquad \textnormal{where }\sigma^2 \defeq \E Z^2\,.
    \]
\end{assumption}
\noindent When the entries of $A$ are i.i.d.\ copies of a random variable satisfying Assumption~\ref{ass:momentGrowthAssumption}, the maximum row- and column-norm parameter
on the right-hand side of
Theorem~\ref{thm:seginerMatrix}
simplifies to\footnote{Throughout, we use the notation $a\asymp_{\eps,K} b$ to mean that $c b\le a\le C b$ for some constants $c=c(\eps,K)$ and $C=C(\eps,K)>0$ depending only
on $\eps$ and $K$. 
This dependence is necessary in general (see Remark~\ref{rem:tightnessSeginer}).}
\[
    \E \max_{i\in [n]} \left(\sum_{j=1}^m A_{ij}^2\right)^{\frac12} + \E \max_{j\in [m]} \left(\sum_{i=1}^n A_{ij}^2\right)^{\frac12}  \underset{\eps,K}\asymp \sigma \sqrt {\max(n,m)} +  \E \max_{\substack{i\in [n]\\j\in [m]}} |A_{ij}|\,.
\]
(See the $r=2$ case of Proposition~\ref{prop:proofSeginerEquivalence} for a self-contained proof of this equivalence.) 

Our next result is a
tensor analogue of this 
consequence of Seginer's theorem.

\begin{theorem}[Seginer-type theorem for tensors; see Theorem~\ref{thm:seginer}]
    \label{thm:seginerIntro}
    Let $ T\in \R^{n_1\times \ldots \times n_r}$
    be a random tensor whose entries
    are i.i.d.\ copies of a random variable satisfying Assumption~\ref{ass:momentGrowthAssumption}. Then,
    \[
        \E \normin{T}\underset{\eps,K,r}\asymp \sigma \cdot \max_{i\in [r]} \sqrt{n_i} + \E \max_{\substack{i_1\in [n_1]\\\ldots\\i_r\in [n_r]}} |T_{i_1,\ldots,i_r}|\,.
    \]
\end{theorem}

\noindent The dependence on $\eps,K,r$ 
in the upper bound is polynomial, whereas the dependence in the lower bound can be removed entirely by replacing $\sigma$ with a slightly different quantity; see Theorem~\ref{thm:seginer} for
a precise statement.

\paragraph*{}Our work leaves open the question of
a full analogue of Theorem~\ref{thm:seginerMatrix}
for tensors whose entries do not
satisfy Assumption~\ref{ass:momentGrowthAssumption}:

\begin{conjecture}\label{conj:seginerTensor}
    Let $T\in \R^{n_1\times \ldots \times n_r}$ be a random tensor
    with i.i.d.\ mean-zero entries.
    Then,
    \[
        \E \normin{T} \underset{r} \asymp \max_{k\in [r]} \E \max_{\substack{i_1\in [n_1]\\\ldots\\i_{r}\in [n_{r}]}} \left( \sum_{j=1}^{n_k} T^2_{i_1, \ldots, i_{k-1}, j, i_{k+1}, \ldots, i_r}\right)^{1/2}\,.
    \]
\end{conjecture}

We prove that Conjecture~\ref{conj:seginerTensor}
is equivalent to Theorem~\ref{thm:seginer}
under Assumption~\ref{ass:momentGrowthAssumption}
(for constant $K,\eps$ independent of $n$)
in Proposition~\ref{prop:proofSeginerEquivalence}.
We note that even the special case of Conjecture~\ref{conj:seginerTensor} where $n_1 = \ldots = n_r=n$ and the entries have
a centered Bernoulli distribution with parameter $1/n$ remains open
and would already be of independent
interest.

In a related direction, Boedihardjo~\cite{marchTensor} recently gave a tensor
generalization of a different
theorem of~\cite{seginer}, which
holds in the setting of independent (but not
necessarily identically distributed) entries; see
Theorem~\ref{thm:march} below. This remarkable result often yields sharp bounds on the norm of random tensors with dense entries, but it is not suited to the sparse setting because its error term involves $\polylog n$
factors. Such logarithmic factors are prohibitive for applications such
as Theorem~\ref{thm:mainTensor},
but they are unavoidable for the general independent entry model even in the matrix case, as shown in~\cite{seginer}. 
Nevertheless, our proof of Theorem~\ref{thm:seginerIntro} crucially relies on the combination of Theorem~\ref{thm:mainTensor} and the results in~\cite{marchTensor}, where Boedihardjo's bound covers the small entries of the tensor and our result controls the outliers.

\subsection{Proof idea: an explicit selector process decomposition}
\label{sec:overview}

The trace method is the classical approach to estimating the spectral norm of a random matrix, but it has no known analogue in the tensor setting.
One way to reduce the problem to the matrix case is to flatten the tensor.
For instance, given $T\in (\R^n)^{\otimes 4}$, one
may bound its injective norm by that of the $n^2\times n^2$ matrix whose entry in position $((i,j),(k,\ell))$ is $T_{ijk\ell}$. However, this reduction is not sufficient for Theorem~\ref{thm:mainTensor},
since the resulting matrix contains a
row with $\omega(1)$ nonzero entries
when
$p=\Omega(n^{-2})$.
Moreover, other common tools that have been used recently to study the injective norm of random tensors~\cite{aden2025injective,2024DartoisMcKenna,dartois2026moment,marchTensor,tensorConcentration,al2025sharp} are tailored to the Gaussian and sub-Gaussian setting. 
The sparsity regime we consider lies outside of the Gaussian universality class, so these approaches cannot be applied without major modifications.

Among the standard approaches to the matrix problem, only the union bound argument
of Kahn and \Szemeredi~\cite{kahnSzemeredi}
appears to generalize. 
The method discretizes the unit sphere by restricting to
vectors with signed dyadic entries, which changes the supremum by at
most a universal constant factor.
The main step is then a technical case analysis establishing the desired union bound for this discretization. 
Although the
approach uses only first principles, its implementation typically
requires a complex and lengthy case analysis. This is already apparent in the
matrix setting and becomes particularly cumbersome for tensors~\cite{kahnSzemeredi,friedmanWigderson,feigeOfek,tensorCompletionJain,
feigeOfekBoosted,sparseTensorZhou}.
This motivates the search for a more systematic and conceptual way to design such arguments.

\paragraph*{}
We instead use the theory of selector processes~\cite[Section
11.11]{talagrandChaining} as a guide for constructing union bound arguments. Given
$\calX\subseteq \R^m$, let
\begin{align}
    \delta(\calX)
  \defeq \E \sup_{X\in \calX} \iprod{T - \E T, X}\,,\label{eq:selector}
\end{align}
where $(T_i)_{i\in [m]}$ are i.i.d.\
Bernoulli random variables. Stochastic processes such as $(\iprod{T - \E T, X})_{X\in \calX}$ are called {\em selector processes}.
Talagrand shows~\cite[Theorem 11.12.1]{talagrandChaining} that
the following strategy gives an upper bound on
$\delta(\calX)$ that is optimal up to universal constants:
\begin{enumerate}
    \item Find a
    decomposition $\calX \subseteq \calX_B+\calX_P$, where we refer to $\calX_B$ as the ``Bernstein part'' and $\calX_P$ as the ``positive part''.

    \item Bound $\delta(\calX_B)$ by a chaining argument
    based only on Bernstein's inequality.

    \item Bound the positive part $\delta(\calX_P)$ after discarding sign
    cancellations:
    \[
        \delta(\calX_P)\le \E \sup_{X\in \calX_P} \iprod{T, X}\,.\footnote{Assuming that $\calX_P=-\calX_P$, as will be the case in our application.}
    \]
     Moreover,~\cite[Research Problem 13.1.1]{talagrandChaining} further suggests the existence of a finite witness set $\calW\subseteq \R^m_{\ge 0}$
    such that
    \[
        \forall X\in \calX_P\,,\;\exists Y\in \operatorname{Conv}(\calW)
            \textnormal{ such that }
            X_i \le Y_i \textnormal{ for all } i\in [m]\,,
    \]
    (where $\operatorname{Conv}$ denotes
    the convex hull),
    and for which a direct union
    bound over $\calW$ yields a tight estimate:
    \[
        \delta(\calX_P) \asymp
        \inf \left\{
            t>0 :
            \int_t^\infty
            \sum_{W\in \calW}
            \Pr\left(\iprod{T, W}>s\right)
            \diff s
            \le t
        \right\}.
    \]
\end{enumerate}
Such a decomposition immediately gives an upper bound on $\delta(\calX)$. The nontrivial fact is that there
always exists a decomposition
providing a bound on $\delta(\calX)$ that is tight up to constant factors. We emphasize that, to the best of our knowledge, the additional assertion concerning the existence of a witness set remains a (far-reaching) conjecture (see~\cite{park2024conjecture} for recent progress on a related question).

Unfortunately, Talagrand's proof
of existence of such a decomposition does not explain how
to choose the decomposition in practice. This limitation is shared by much
of the chaining literature. For instance, it is known
abstractly that the supremum of a Gaussian process can always be bounded
optimally using a generic chaining construction.
Nevertheless, for most Gaussian
processes arising in random matrix theory applications it remains open to exhibit an explicit construction with this property~\cite{talagrandChaining,vanHandelChaining,leeChaining}. A partial exception is the recent
work~\cite{latalaRademacher}, which relies on the decomposition theorem for Bernoulli processes proved in~\cite{bernoulliConjecture} as a guide
for the analysis of the norm of Rademacher matrices.

In this paper, we implement this strategy explicitly for our selector process~\eqref{eq:selector}. We prove Theorems~\ref{thm:mainTensor} and~\ref{thm:tensorVersion}
by constructing a decomposition together with
a witness set for the positive part.
Our goal is to bound the expected supremum of the selector process over
\begin{equation}
    \mathcal X \defeq \{x_1\otimes \ldots \otimes x_r : \|x_1\|_2= \ldots= \|x_r\|_2 = 1\}\,.\label{eq:pureTensors}
\end{equation}

To define the Bernstein part, we clip every entry
of a test tensor $X\in \calX$
at the level $1/(n \log n)$. The remaining
contribution forms the positive part; see~\eqref{eq:decomposition} for the formal decomposition. This decomposition plays the role of the small-pair/large-pair
decomposition in the traditional Kahn-\Szemeredi argument.

The witness set consists of suitably normalized pure tensors whose factors are $\{0,1\}$-valued vectors; see Definition~\ref{def:witness}. 
The general idea of transferring bounds from $\{0,1\}$-valued
test vectors to arbitrary test vectors
also appears in the ad hoc discretization
arguments~\cite{kahnSzemeredi,feigeOfek}.
Such a lifting generally requires a carefully
chosen normalization of the indicator
vectors, because a naive reduction incurs
a loss depending on the largest $\ell_1$-norm
of a tensor slice; see~\cite[Section 4.2]{biluLinial} for the matrix case.
Identifying the appropriate normalization
is the main substantive step of our proof;
once this is in place, the rest of the proof
follows directly from the above strategy.

\subsection{Organization of the paper}

After the preliminaries in Section~\ref{sec:prelims}, we prove
the general sparse tensor result underlying Theorem~\ref{thm:mainTensor} (Theorem~\ref{thm:mainTensorProof}) in
Section~\ref{sec:mainProof}. In particular,
we discuss some applications of
this result in Section~\ref{sec:applis}.
We derive some extensions, including Theorem~\ref{thm:tensorVersion},
in Section~\ref{sec:extensions}.
In Section~\ref{sec:seginer}, we
prove our generalization of Seginer's theorem (Theorem~\ref{thm:seginerIntro}).
Appendix~\ref{sec:truncationinequalities} contains omitted lemmas and their proofs.
Appendix~\ref{sec:equivalenceSeginer}
shows the equivalence between Conjecture~\ref{conj:seginerTensor}
and Theorem~\ref{thm:seginerIntro} under
our moment equivalence assumption.

\subsection{Acknowledgments}

LP thanks Luca Trevisan for suggesting, several years ago, the problem of determining the spectral gap of \Erdos-\Renyi hypergraphs. LP's work on this
project was supported by the Swiss National Science Foundation, grant no.\ 10004947.

The authors used AI tools for literature research and to scan the manuscript for typos. All content was written and verified by the authors.

\section{Preliminaries}
\label{sec:prelims}

\paragraph{Asymptotic notation:} We write $a\lesssim b$ (resp. $a\gtrsim b$) to denote
$a\le Cb$ (resp. $a\ge Cb$) for some universal constant $C>0$. We write
$a\asymp b$ as a shortcut for
$a\lesssim b$ and $a\gtrsim b$. We
write $a\lesssim_{s} b$ when the
hidden constant $C$ is allowed to
depend on $s$.

\paragraph{Tensor notation:} For tensors $T,T'\in \R^{n_1\times \ldots n_r}$, write $T\le T'$ when $T_i\le T'_i$ for every index $i\in [n_1]\times\ldots\times[n_r]$. Denote by $|T|$ the tensor obtained
from $T$ by applying the absolute
value entrywise. For $p\in [1, \infty]$, 
write $\|T\|_p$ for the
$\ell_p$-norm of the vectorization of $T$, viewed as an element of
 $\R^N$ for $N=\prod_{i=1}^r n_i$.

\begin{definition}[Clip]
    Given $\tau\in \R_{\ge 0}$, 
    let $\clip_\tau:\R\to \R$ be
    defined by
    \[
        \clip_\tau(x) =\begin{cases}
            x & \text{ if $|x|\le\tau$,}\\ \tau & \text{ if $x>\tau$,}\\
            -\tau & \text{ if $x<-\tau$.}
        \end{cases}
    \]
    For a tensor $T\in \R^{n_1\times\ldots n_r}$, we extend
    the notation by letting
    $\clip_\tau(T)$ be the tensor
    obtained from $T$
    by applying $\clip_\tau$
    entrywise.
\end{definition}

\paragraph{Concentration inequalities:} We use only
two elementary concentration
inequalities: Bernstein's inequality
and the multiplicative Chernoff bound.

\begin{lemma}[Bernstein's inequality; see, e.g., {\cite[Lemma 4.5.6]{talagrandChaining}}]\label{lem:bernstein}
    Let $X_1, \ldots, X_N$ be
    independent mean-zero random variables
    such that $|X_i|\le M$ almost
    surely. Then for any $t\ge 0$,
    \[
        \Pr\left(\sum_{i=1}^N X_i >t\right)\le \exp\left(-\frac 14 \min\left(\frac {t^2}{\sum_{i=1}^N \E X_i^2}, \frac t {M}\right)\right)\,.
    \]
\end{lemma}

Let $\textnormal{Ber}(p)$ denote
the Bernoulli distribution with parameter $p\in (0,1)$. 

\begin{lemma}[Chernoff bound; see, e.g., {\cite[Theorem 2.3.1]{vershynin2026high}}]\label{lemma:chernoff}
    Let $X_1, \ldots, X_N$ be i.i.d.\ $\operatorname{Ber}(p)$ random variables. Then, for any  $t \ge e^2\cdot Np$, 
    \[
        \Pr\left(\sum_{i=1}^N X_i > t\right) \leq \exp \left(-\frac{t}{2}\log \left( \frac{t}{Np} \right) \right)\,.
    \]
\end{lemma}

\paragraph{Solid convex hull:}
Given $\calX \subseteq (\R^n)^{\otimes r}$, we 
denote by $\operatorname{Conv}(\calX)$ the convex hull of $\calX$
and by
\[
    \solid(\calX) \defeq \{X\in (\R^n)^{\otimes r} : \exists Y \in \operatorname{Conv}(\calX) \text{ s.t. } X\le Y\}
\]
its solid convex hull~\cite{talagrandChaining}.

\section{Tensors with sparse independent entries}
\label{sec:mainProof}

In this section, we prove a general bound on the norm of sparse random tensors. Section~\ref{sec:applis} presents applications of the theorem to the settings introduced above. We then prove the theorem in Sections~\ref{sec:proofFirst} through~\ref{sec:proofLast}.

\begin{theorem}[Bound on the injective norm of sparse random tensors]\label{thm:mainTensorProof}
    Let $n\ge 5$, $r\ge 2$ and $\eps\in (0,1)$ such that
    \begin{equation}
        r\le \frac {\eps \log n}{5\log \log n}\,.\label{eq:rAssumption}
    \end{equation}
    Let $T\in (\R^n)^{\otimes r}$ be a random tensor with independent entries taking values in $[-1,1]$, and suppose that for some
    $L\geq 1$, we have $\Pr\left(T_i\neq 0\right)\le Ln^{-1-\eps}$ for all $i\in [n]^r$.
    Then,
    \[
        \E \normin{T - \E T} \lesssim \frac {Lr^4} \eps\,.
    \]
    Moreover, for any $t>0$,
    \[
        \Pr\left(\normin{T - \E T}\gtrsim \frac {Lr^3} \eps\left(r + t\right)\right)\le n^{-t}\,.
    \]
\end{theorem}

\subsection{Applications \texorpdfstring{of Theorem~\ref{thm:mainTensorProof}}{}}
\label{sec:applis}

Before proving Theorem~\ref{thm:mainTensorProof}, we give several applications and comparisons with prior work.
First, if each entry of $T$ is distributed as $\textnormal{Ber}(n^{-1-\eps})$, Theorem~\ref{thm:mainTensorProof} recovers Theorem~\ref{thm:mainTensor} as a corollary. 

\begin{proof}[Proof of Theorem~\ref{thm:mainTensor} from Theorem~\ref{thm:mainTensorProof}]
    Let $T\sim \calH_r(n,p)$. The upper bound \[\normin{T} - \normin{T - \E T}\le \normin{\E T}\le pn^{r/2}\]is immediate. For the lower bound, first, applying the
    all-ones
    test vector shows that
    \[
        \normin{T}\ge \frac 1 {n^{r/2}} \sum_{i\in [n]^r} T_i\,.
    \]
    By Lemma~\ref{lem:bernstein}, the right-hand side is at least $(1-o(1))pn^{r/2}$
    with high probability. It remains
    to apply Theorem~\ref{thm:mainTensorProof} to $\normin{T - \E T}$. 
    
    Although the entries of an \Erdos-\Renyi adjacency tensor are not independent (because of the symmetry), we can partition its entries according to the ordering of their indices, writing the tensor as a sum of $r!$ tensors with independent entries. Since $r$ is fixed,~\eqref{eq:rAssumption} holds
    for sufficiently large $n$.
    Therefore, setting $L=\max(1, pn^{1.01})$, we can apply Theorem~\ref{thm:mainTensorProof} and we obtain, with high probability,
    \[
        \normin{T - \E T}\lesssim_r L = o(pn^{r/2})\,,
    \]
    whenever $r\ge 3$ and $pn^{r/2}\to\infty$. This concludes the proof.
\end{proof}

\begin{remark}    
    The high probability bound also shows
    that the conclusion of Theorem~\ref{thm:mainTensor} still holds for the variant of the
    \Erdos-\Renyi model in which $m$ hyperedges are selected
    uniformly at random. This follows by conditioning the model in Definition~\ref{def:erdos}
    on having exactly $m=p\binom {n} r$ hyperedges,
    since this event has probability
    at least $\Omega(m^{-1/2})$.
\end{remark}

\begin{example}[Expander mixing lemma]\label{ex:mixingLemma}
    Applying the injective norm bound
    to $\{0,1\}$-valued test vectors
    directly yields an improved expander
    mixing lemma for sparse \Erdos-\Renyi
    hypergraphs. More precisely, for every
    $\eps\in (0,1)$, with high probability over $T\sim \calH_r(n,p)$, the following holds
    simultaneously
    for every collection of pairwise disjoint sets $U_1, \ldots, U_r\subseteq [n]$:
    \[
        \Abs{e_T(U_1, \ldots, U_r) - p \prod_{i=1}^r \Abs{U_i}}\lesssim_{r,\eps} (1+pn^{1+\eps})\prod_{i=1}^r \sqrt{\Abs{U_i}}\,,
    \]
    where $e_T(U_1, \ldots, U_r)$ denotes the
    number of hyperedges containing
    exactly one vertex from each of $U_1, \ldots, U_r$.
\end{example}

\begin{example}[Tensor completion]
    Given a
    deterministic tensor $T\in (\R^n)^{\otimes 3}$, we form a random set $\Omega\subseteq [n]^3$
    of entries by including each entry independently
    with probability $p \defeq \alpha n^{-3/2}$ for some $\alpha\ge 1$.
    Then, we
    reveal the entries of $T$ in $\Omega$. Denote by $P_\Omega(T)$ the observed tensor (i.e., $P_\Omega(T)_i = T_i$ if $i\in \Omega$,
    and $P_\Omega(T)_i=0$ otherwise). Note that $\frac 1p P_\Omega(T)$ is an unbiased
    estimator of $T$. The question
    is how accurately this estimator approximates
    $T$ in injective norm. 
    
    This question arises, for example, in the analysis
    of algorithms for tensor completion,
    where the goal is to reconstruct $T$ from $P_\Omega(T)$. In their analysis of the alternating minimization algorithm for this task, Jain and Oh~\cite[Appendix A]{tensorCompletionJain} prove the estimate
    \begin{equation*}
        \frac 1 {\Normi{T} n^{3/2}} \E \normin{\frac 1p P_\Omega(T) - T}\lesssim \frac{\log^2 n} {\sqrt \alpha}\,.
    \end{equation*}
    In this setting,
    Theorem~\ref{thm:mainTensorProof}
    implies that if $p\le n^{-1.01}$, then
    \begin{equation*}
        \frac 1 {\Normi{T}  n^{3/2}} \E \normin{\frac 1p P_\Omega(T) - T}\lesssim \frac{1} {\alpha}\,.
    \end{equation*}
\end{example}

\begin{example}[Multilayer community detection]
    In the stochastic block model,
    one considers $n$ individuals and
    an unknown community assignment $\sigma:[n]\to [k]$ of the individuals to $k$ communities. In the multilayer version of the problem, we observe $L$ undirected graphs (layers)
    $G_1, \ldots, G_L$ on the same vertex set $[n]$. In $G_\ell$, an edge between $u$
    and $v$ appears with probability $p_\ell(\sigma(u), \sigma(v))$, independently over all $\ell\in [L]$ and all
    $1\le u<v\le n$. Given $G_1, \ldots, G_L$,
    the goal is to recover $\sigma$ (up to a
    permutation of the community labels).
    
    The observations can be collected in an order-$3$
    tensor $T\in \R^{L\times n\times n}$,
    where $T_{\ell uv} = 1$ if $\{u,v\}$ is an edge in $G_\ell$, and $T_{\ell uv} = 0$ otherwise. 
    Lei, Chen, and Lynch~\cite{multiLayerSBM} analyze a least-squares estimator for this problem.
    A key ingredient in their analysis is an
    estimate for the concentration of $T$ around its expectation. When $L\le n$, they show~\cite[Theorem 2]{multiLayerSBM} that
    \[
        \normin{T - \E T}\lesssim \log^{3/2} n + \sqrt{n\rho } \log n\,,\qquad \rho \defeq \max_{\ell \in [L], a,b\in [k]} p_\ell(a,b)\,.
    \]
    In contrast, if $\rho\le n^{-1.01}$, then
    Theorem~\ref{thm:mainTensorProof} gives the stronger bound $\normin{T - \E T} \le O(1)$ with 
    high probability.
\end{example}

\subsection{Proof of \texorpdfstring{Theorem~\ref{thm:mainTensorProof}}{main theorem}}
\label{sec:proofFirst}

Recall that $\calX$ denotes the set of
unit pure tensors in~\eqref{eq:pureTensors}.
We decompose $\calX$ into
\begin{equation}
    \calX_B \defeq \{\clip_\tau(X) : X\in \calX\} \text{ and }\calX_P \defeq \{X - \clip_\tau(X) : X\in \calX\}\,,\label{eq:decomposition}
\end{equation}
for the threshold $\tau\defeq\frac 1 {n\log n}$.
Clearly, we have $\calX\subseteq \calX_P + \calX_B$ by construction. As in Section~\ref{sec:overview}, we refer to $\calX_B$
as the ``Bernstein part'' of the process and to $\calX_P$
as its ``positive part''.
The following lemmas bound the two suprema separately.

\begin{lemma}[Bounding the Bernstein part]\label{lem:bernsteinPart}
    Let $r,n,\eps,L$, and $T\in (\R^n)^{\otimes r}$
    satisfy the assumptions of Theorem~\ref{thm:mainTensorProof}. Then,
    for any $t>0$,
    \[
        \Pr\left(\sup_{X\in \calX_B} \iprod{T - \E T, X} \gtrsim \sqrt L(r^2 + t)\right)\le e^{-tn\log n}\,.
    \]
\end{lemma}

\begin{lemma}[Bounding the positive part]\label{lem:positivePart}
    Let $T\in (\R^n)^{\otimes r}$ be a random tensor with i.i.d.\ $\textnormal{Ber}(Ln^{-1-\eps})$ entries, where $r,n,\eps,L$
    satisfy the assumptions of Theorem~\ref{thm:mainTensorProof}. Then,
    for any $t>0$,
    \[
        \Pr\left(\sup_{X\in \calX_P} \iprod{T, X}\gtrsim \frac {Lr^3} \eps (r + t)\right)\le n^{-t}\,.
    \]
\end{lemma}

\begin{proof}[Proof of Theorem~\ref{thm:mainTensorProof} from Lemmas~\ref{lem:bernsteinPart} and~\ref{lem:positivePart}]
    Lemma~\ref{lem:bernsteinPart} 
    directly bounds
    the supremum of the Bernstein part
    of the original process.
    For the positive part, note that every $X\in \calX_P$ has at most $n^2 \log^2 n$ nonzero entries. Therefore 
    \[
        \abs{\iprod{\E T, X}}\le \norm{\E T}_\infty \norm{X}_1\le Ln^{-1-\eps} n \log n = L  n^{-\eps} \log n\overset{\eqref{eq:rAssumption}} \le \frac {Lr^4} \eps\,,
    \]
    so it suffices to bound the uncentered process
    $\sup_{X\in \calX_P} \langle T, X\rangle$. Since $\calX_P = -\calX_P$, we may furthermore assume, without loss of generality, that the entries of $T$ are supported on
    $[0,1]$.
    Among all distributions on $[0,1]$ satisfying
    the assumptions of Theorem~\ref{thm:mainTensorProof},
    the probability $\Pr\left(\sup_{X\in \calX_P} \langle T, X\rangle>t\right)$ is maximized when the entries of $T$ are distributed as $\textnormal{Ber}(Ln^{-1-\eps})$ (for every $t$).
    Hence, the tail bound in Theorem~\ref{thm:mainTensorProof}
    follows from Lemmas~\ref{lem:bernsteinPart}
    and~\ref{lem:positivePart}, and integrating this tail bound yields the expectation bound.
\end{proof}

The rest of the section
is devoted to the proof of
Lemmas~\ref{lem:bernsteinPart}
and~\ref{lem:positivePart}.

%% file: BernsteinPartVJ26.tex
\subsection{Bernstein part \texorpdfstring{(proof of Lemma~\ref{lem:bernsteinPart})}{}}

First,
for every fixed $X\in \calX_B$, every $C\geq 1$, and
every $t>0$, Bernstein's inequality (Lemma~\ref{lem:bernstein}) gives
\begin{align*}
    \Pr\left(\iprod{T - \E T, X}> C\sqrt L(r^2+t)\right)
    &\le
    \exp\left(
        -\frac 14 \min\left\{
            \frac {C^2(r^2+t)^2} {n^{-1-\eps}},
            \frac{C \sqrt L(r^2+t)} {2\tau}
        \right\}
    \right)\nonumber\\
    &\overset{\eqref{eq:rAssumption}}\le e^{-\frac C 8 (r^2+t) n\log n}\,.
\end{align*}
Next, we show that a net of $2^{O(r^2 n\log n)}$ test vectors suffices to control the supremum of the process over $\calX_B$. Indeed, for any
unit vectors $x_1, \ldots, x_r$
and $y_1, \ldots, y_r$, letting
$X \defeq \bigotimes_{i\in [r]} x_i$
and $Y \defeq \bigotimes_{i\in [r]} y_i$,
\begin{align*}
    \left|
    \iprod{T - \E T, \clip_\tau(X) - \clip_\tau(Y)}
    \right|
    &\le \norm{T - \E T}_2
    \norm{\clip_\tau(X) - \clip_\tau(Y)}_2\tag*{(Cauchy-Schwarz)} \\
    &\le \norm{T - \E T}_2 \norm{X-Y}_2 \tag*{($\clip_\tau$ is 1-Lipschitz)}\\
    &\le 2 n^{\frac r2} \norm{X-Y}_2 \tag*{($|T_i|\le 1$)}\\
    &\le 2rn^{\frac r2}
    \max_{i\in [r]} \|x_i-y_i\|_2\,. \tag*{(triangle inequality and $x_i,y_i$ unit)} 
\end{align*}
Let $\calN$ be a $\frac 1 {2r n^{r/2}}$-net of the unit sphere of $\R^n$; we
may choose it so that $|\calN|\le O(rn^{r/2})^n$ using standard constructions. Define $\calS \defeq \{\clip_\tau(y_1 \otimes \ldots\otimes y_r) : y_1, \ldots, y_r\in \calN\}$.
By the above inequality, for any $X\in \calX_B$, there exists $Y\in \calS$ such that
\[
    \abs{\iprod{T - \E T, X - Y}}\le 1\le r^2 \sqrt L\,.
\]
Moreover, since $|\calS|\le O(rn^{r/2})^{nr}= 2^{O(r^2 n \log n)}$, 
the above tail bound holds simultaneously
for all vectors in the net, provided that $C$ is sufficiently large. Lemma~\ref{lem:bernsteinPart}
then follows.

%% file: WitnessConstructionVJ26.tex
\subsection{Positive part: witness set construction}

We begin by constructing a witness
set for $\calX_P$ (see Section~\ref{sec:overview} for background). We
use the following
normalization for the indicator $\mathbf 1_S$
of a non-empty set $S\subseteq [n]$:
\[
    \phi(S) \defeq \frac {\sqrt{|S|}} {\max(4r^2, \log^2|S|)}\,.
\]
We will repeatedly use the following elementary consequence of~\eqref{eq:rAssumption} and $n\ge 5$:
\begin{equation}
    \phi(S)\ge \frac {\sqrt {|S|}}{\log^2 n}\,.\label{eq:lowerBoundPhi}
\end{equation}

\begin{definition}[Witness set for $\calX_P$]\label{def:witness}
    Define the witness set $\calW= \calW_0 \cup \ldots \cup \calW_r$ as follows. First, let $\calW_0$ be the set of
    balanced indicator witnesses:
    \begin{align*}
        \calW_0 &\defeq \left\{2\bigotimes_{i\in [r]} \frac{\mathbf 1_{A_i}} {\phi(A_i)} : A_1, \ldots, A_r\textnormal{ s.t. } \frac 1r \sum_{i=1}^r |A_i| \le \prod_{i=1}^r \phi(A_i) \text{ and } \prod_{i=1}^r |A_i|\le r^{2r} n^2\log^2 n\right\}\,,
    \end{align*}
    where $A_1, \ldots, A_r$ range over
    all non-empty subsets of $[n]$ 
    satisfying the conditions. Second, for each $j\in [r]$,  let $\calW_j$ be the set of unbalanced indicator witnesses whose $j$-th coordinate is unconstrained:
    \begin{align*}
        \calW_j &\defeq \left\{\frac {\bigotimes\limits_{i=1}^{j-1}\mathbf 1_{A_i}\otimes  \mathbf 1_{[n]} \otimes \bigotimes\limits_{i=j+1}^{r}\mathbf 1_{A_i} } {\frac 1 {24 r^4}\sum\limits_{\substack{i=1\\i\neq j}}^r |A_i| + \frac 1 {4r \log^4 n}\prod\limits_{\substack{i=1\\i\neq j}}^r \phi(A_i)^2} : A_1,\ldots,A_{j-1}, A_{j+1},\ldots,A_r\subseteq [n]\text{ non-empty}\right\}\,.
    \end{align*}
\end{definition}

We show that $\calW$ is a valid
witness set for $\calX_P$.

\begin{lemma} \label{lem:solidContain}
    We have
    \[  
        \calX_P \subseteq \solid\left(\bigcup\limits_{j=0}^r \calW_j\right)\,.
    \]
\end{lemma}

\begin{proof}
    It suffices to consider
    vectors with non-negative entries. Let $x_1, \ldots, x_r\in \R^n_{\ge 0}$ be unit vectors.
    Given
    $k=(k_1, \ldots, k_r)\in\N^r$ and $i\in [r]$, write
    $A_{k,i} \defeq \{j\in [n] : (x_i)_j\in [r^{-k_i-1}, r^{-k_i}]\}$.
    Since each $x_i$ is a unit vector,
    \begin{equation}
        \sum_{k\ge 0} |A_{k,i}| r^{-2k_i-2} \le 1\implies |A_{k,i}|\le r^{2k_i+2}\,.\label{eq:counting}
    \end{equation}
    Taking tensor products of this level set
    decomposition gives
    \begin{equation}
        \bigotimes_{i\in [r]} x_i - \clip_\tau\left(\bigotimes_{i\in [r]} x_i\right)\le \sum_{\substack{k_1, \ldots, k_r\ge 0\\\prod\limits_{i\in [r]}r^{-k_i}>\tau}} \bigotimes_{i\in [r]} \frac{\mathbf 1_{A_{k,i}}}{r^{k_i}}\,.\label{eq:mainDiscret}
    \end{equation}
    We cover the indices of the sum on the right-hand side of~\eqref{eq:mainDiscret} by
    \begin{align*}
        \calK_{0} &= \left\{(k_1, \ldots, k_r)\in \N^r:
        \frac 1r \sum_{i=1}^r |A_{k,i}|\le \prod_{i=1}^r \phi(A_{k,i})\text{ and } \prod_{i=1}^r |A_{k,i}|\le r^{2r} n^2\log^2 n\right\}\,,\\
        \calK_{j} &= \left\{(k_1, \ldots, k_r)\in \N^r : \frac 1r \sum_{i=1}^r |A_{k,i}| > \prod_{i=1}^r \phi(A_{k,i}) \text{ and } |A_{k,j}| = \max_{i\in [r]} |A_{k,i}|\right\}\quad \forall j\in [r]\,.
    \end{align*}
    First, we claim that 
    the sets $\calK_0,\ldots,\calK_r$
    cover every index appearing in the sum on 
    the right-hand side
    of~\eqref{eq:mainDiscret}. Indeed, by~\eqref{eq:counting}, we have
    \[
        \prod_{i\in [r]} r^{-k_i}>\tau \implies \prod_{i=1}^r |A_{k,i}|\le r^{2r} \prod_{i=1}^r r^{2k_i}\le \frac {r^{2r}} {\tau^2} = r^{2r} n^2 \log^2 n\,,
    \]
    so the second condition
    in the definition of $\calK_0$
    is automatically satisfied.

    Second, we show that, for each $j\in \{0, \ldots, r\}$, the sum over $\calK_j$
    lies in the solid convex hull
    of the corresponding $\calW_j$.
    
    \begin{lemma}\label{lemma:bSum}
    We have
    \[
        \sum_{k\in \calK_0} \bigotimes_{i\in [r]} \frac {\mathbf 1_{A_{k,i}}} {r^{k_i}} \in \frac12 \solid(\calW_0)\,.
    \]
    \end{lemma}
    
    \begin{lemma}\label{lemma:uSum}
        For any $j\in [r]$, we have
        \[\sum_{k\in \calK_j} \bigotimes_{i\in [r]} \frac {\mathbf 1_{A_{k,i}}} {r^{k_i}}  \in \frac 1{2r}  \solid(\calW_j)\,.
        \]
    \end{lemma}
    
    \noindent Lemmas~\ref{lemma:bSum}
    and~\ref{lemma:uSum} complete the proof of Lemma~\ref{lem:solidContain}.
\end{proof}

\begin{proof}[Proof of Lemma~\ref{lemma:bSum}]
    After an appropriate rescaling, each tensor $\bigotimes_i \mathbf 1_{A_{k,i}}$ appearing in the
    sum
    is an element of $\calW_0$. Thus, we can expand the sum as a linear combination of elements of $\calW_0$,
    with total coefficient
    \begin{align*}
        \frac 12 \sum_{k\in \calK_0} \prod_{i=1}^r r^{-k_i} \phi(A_{k,i}) &\le \frac12\sum_{k_1, \ldots, k_r\ge 0} \prod_{i=1}^r \frac r {\max(4r^2, (2k_i+2)^2\log^2 r)}\,,
        \intertext{where we used~\eqref{eq:counting} and the fact that $x\mapsto \frac {\sqrt x} {\max(4r^2, \,\log^2 x)}$
    is increasing on $[1,\infty)$. The sum
    factorizes, yielding}
        &\le \frac 12 \left(\sum_{k\ge 1} \min\left(\frac 1{4r}, \frac r {4k^2 \log^2 r}\right)\right)^r\,.
    \end{align*}
    Let $k^* \defeq \frac r {\log r}$ denote the crossover value at which the
    second term achieves the minimum.
    By a sum-integral comparison,
    \begin{align*}
        \sum_{k\ge 1} \min\left(\frac 1{4r}, \frac r {4\log^2 r}\cdot \frac 1 {k^2}\right)
        &\le \frac {k^*} {4r} + \frac r {4\log^2 r} \int_{k^*}^\infty \frac 1 {x^2} \,\mathrm dx = \frac 1 {2\log r}<1\,,
    \end{align*}
    since $r\ge 2$.
    Hence, the sum of the coefficients is at most $1/2$, concluding the argument.
\end{proof}

\begin{proof}[Proof of Lemma~\ref{lemma:uSum}]
    Without loss of generality,
    assume $j=1$. As in Lemma~\ref{lemma:bSum}, we bound the total coefficient in the natural expansion of the left-hand side in terms of vectors in $\calW_1$, which is
    \[
        \sum_{k\in \calK_1} \prod_{i=1}^r r^{-k_i}\left(\frac 1 {24r^4}\sum_{i=2}^r |A_{k,i}| + \frac 1 {4r\log^4 n}\prod_{i=2}^r \phi(A_{k,i})^2\right)\,.
    \]
For the second term, since $\log^4 n\ge \max(4r^2, \log^2 |A_{k,1}|)^2$ by~\eqref{eq:rAssumption}, we have
\[
    \frac 1 {\log^4 n}\prod_{i=2}^r \phi(A_{k,i})^2 \le \frac{\prod_{i=1}^r \phi(A_{k,i})^2} {|A_{k,1}|}\le \frac {\prod_{i=1}^r \phi(A_{k,i}) \cdot \frac 1r \sum_{i=1}^r |A_{k,i}|} {|A_{k,1}|} \le \prod_{i=1}^r \phi(A_{k,i})\,.
\]
Therefore, using the same argument as in Lemma~\ref{lemma:bSum},
\[
    \frac 1{4r\log^4 n} \sum_{k\in \calK_1} \prod_{i=1}^r r^{-k_i} \prod_{i=2}^r \phi(A_{k,i})^2\le \frac 1 {4r} \sum_{k_1, \ldots k_r\ge 0} \prod_{i=1}^r  r^{-k_i} \phi(A_{k,i})<\frac 1 {4r}\,.
\]
It remains
to bound the first term,
\begin{align*}
    \sum_{k\in \calK_1} \prod_{i=1}^r r^{-k_i}\sum_{i=2}^r |A_{k,i}|
    &\le 
        \sum_{k_1\ge 0} r^{-k_1}\sum_{i=2}^r \sum_{\substack{k_i\ge 0\\|A_{k,i}|\le |A_{k,1}|}} |A_{k,i}| r^{-k_i} \prod_{\substack{\ell=2\\\ell\neq i}}^r \sum_{k_\ell\ge 0} r^{-k_\ell}\\
        &\le r \left(\frac 1 {1-\frac 1 r}\right)^{r-2} \max_{i=2,\ldots,r}\sum_{\substack{k_1,k_i\ge 0\\|A_{k,i}|\le |A_{k,1}|}} r^{-k_1-k_i}  |A_{k,i}|\,.
    \end{align*}
    Finally, for each $i\in \{2, \ldots, r\}$, 
    split the sum according to whether $k_i\ge k_1$
    or $k_1\ge k_i$:
    \begin{align*}
        \sum_{\substack{k_1,k_i\ge 0}} r^{-k_1-k_i}  \min(|A_{k,1}|,|A_{k,i}|) &\le \sum_{k_1\ge 0} r^{-2k_1} |A_{k,1}| \sum_{k_i\ge k_1} r^{-(k_i-k_1)} + \sum_{k_i\ge 0} r^{-2k_i} |A_{k,i}| \sum_{k_1\ge k_i} r^{-(k_1-k_i)}\\
        &\le 2 \cdot r^2 \cdot \frac 1 {1-\frac 1 r}\,.
    \end{align*}
    Using $\left(\frac {r} {r-1}\right)^{r-1}\le 3$, we get that
    the
    sum of the coefficients is
    at most $\frac 1 {2r}$, as desired.
\end{proof}

%% file: PositivePartVJ26.tex
\subsection{Positive part: union bound \texorpdfstring{(proof of Lemma~\ref{lem:positivePart})}{}}
\label{sec:proofLast}

Having constructed a witness set for $\calX_P$, it remains to apply a union bound over the witnesses.

\begin{lemma}\label{lem:0tail}
    Let $T$ be as in Lemma~\ref{lem:positivePart}. For all $t\ge 1$,
    \[
        \Pr\left(\max_{W\in \calW_0} \Iprod{T, W}\gtrsim \frac {tLr} \eps\right)\le n^{-t}\,.
    \]
\end{lemma}

\begin{proof}
    Let $A_1, \ldots, A_r \subseteq [n]$ satisfy the
    conditions in the definition of
    $\calW_0$.
    We apply the Chernoff bound (Lemma~\ref{lemma:chernoff}):
    for all sufficiently large $t$,
    \begin{align*}
        \Pr\left(\Iprod{T, \bigotimes_{i\in [r]} \frac {\mathbf 1_{A_i}} {\phi(A_i)}} > \frac {tLr} \eps\right)&\le \exp\left(-\frac {tr \prod\limits_{i=1}^r \phi(A_i)}{2\eps} \log\left(\frac { tLr \prod\limits_{i=1}^r \phi(A_i)} { \eps  n^{-1-\eps}\prod\limits_{i=1}^r |A_i| }\right)\right)\,.
    \end{align*}
    Note that the assumption
    of Lemma~\ref{lemma:chernoff}
    is verified, as
    \[
        \frac{tLr \prod\limits_{i=1}^r \phi(A_i)} {\eps  n^{-1-\eps} \prod\limits_{i=1}^r |A_i|} \overset{\eqref{eq:lowerBoundPhi}} \ge \frac {tr n^{1+\eps}} \eps \cdot \frac {1} {\log^{2r} n\prod\limits_{i=1}^r \sqrt {|A_i|} }\ge  {t n^{\eps}}  \cdot \frac {1} {r^r \log^{2r+1} n}\overset{\eqref{eq:rAssumption}}\ge tn^{\frac {3\eps} {10}}\,.
    \]
    Therefore,
    the assumption holds provided
    that $t$ is sufficiently large. 
    Next, using the assumptions on $A_i$,
    we can simplify the upper bound to
    \begin{align*}
        \Pr\left(\Iprod{T, \bigotimes_{i\in [r]} \frac {\mathbf 1_{A_i}} {\phi(A_i)}} > \frac{tLr}\eps\right)&\le \exp\left(-\frac {t\sum\limits_{i=1}^r |A_i|} {2\eps} \log\left(t n^{\frac{3 \eps}{10} }\right)\right)\\
        &\le \exp\left(- \frac {3t} {20} \sum\limits_{i=1}^r |A_i|\log n\right)\,,
    \end{align*}
    using the same estimate as above.    
    Finally, the number of witnesses in $\calW_0$ corresponding to a prescribed $r$-tuple of set sizes $|A_1|, \ldots, |A_r|$ is at most $\prod_{i=1}^r \binom n {|A_i|}$, so we can take a union bound and use Lemma~\ref{lemma:subsetunionboundseries} provided that $t$ is large enough. Increasing the implicit universal constant extends the result to all $t\ge 1$.
\end{proof}

\begin{lemma}\label{lem:jTail}
    Let $j\in [r]$ and $T$ be as in Lemma~\ref{lem:positivePart}.
    Then, for any $t>0$,
    \[
        \Pr\left(\max_{W\in \calW_j} \Iprod{T, W}\gtrsim \frac {Lr^3} \eps(r+t)\right)\le n^{-t}\,.
    \]
\end{lemma}

\begin{proof}
    Without loss of generality, assume $j=1$. Let $A_2, \ldots, A_r\subseteq [n]$ and denote by
    \[
        \Delta \defeq \frac 1{24r^4} \sum_{i=2}^r |A_i| + \frac 1{4r\log^4 n}  \prod_{i=2}^r \phi(A_i)^2\,.
    \]
    Then, by the Chernoff bound (Lemma~\ref{lemma:chernoff}), we have, for any $t>0$ and for a sufficiently large universal constant $C$,
    \[
        \Pr\left(\Iprod{T, \frac {\mathbf 1_{[n]} \otimes \bigotimes\limits_{i=2}^r \mathbf 1_{A_i}} {\Delta}} \ge C\frac {Lr^3} \eps \left(r +t\right) \right)\le \exp\left(-\frac {C r^3(r+t) \Delta} {2\eps} \log \left(\frac {C\Delta Lr^3(r+t)} {\eps n^{- \eps}\prod\limits_{i=2}^r  |A_i|} \right) \right)\,.
    \]
    Note that the assumption of Lemma~\ref{lemma:chernoff} is verified, as
    \[
        \frac {C\Delta Lr^3(r+t)} {\eps n^{-\eps}\prod\limits_{i=2}^r  |A_i|}\overset{\eqref{eq:lowerBoundPhi}} {\ge} \frac {Cr^2(r+t) n^{\eps}} {4\eps\log^{4r} n}\overset{\eqref{eq:rAssumption}}{\ge} \frac {Cn^{\frac \eps5}} {4}\,.
    \]
    For $C$ sufficiently large, this lower bound exceeds both a universal constant and  $n^{\frac \eps5}$. Therefore,
    the above tail bound implies
    \begin{align*}
        \Pr\left(\Iprod{T, \frac {\mathbf 1_{[n]} \otimes \bigotimes\limits_{i=2}^r \mathbf 1_{A_i}} {\Delta}}\ge  C\frac {Lr^3} \eps(r+t) \right)&\le n^{-\frac {C} {10}r^3(r+t) \Delta}
        \le n^{-\frac {C} {240} \sum\limits_{i=2}^r |A_i|} \cdot n^{-\frac {C t} {480}}\,.
    \end{align*}
    Finally, choosing $C$ sufficiently large,
    we can take a union bound over $\calW_j$, and use Lemma~\ref{lemma:subsetunionboundseries} as in the previous lemma to 
    obtain the desired statement.
\end{proof}

\begin{proof}[Proof of Lemma~\ref{lem:positivePart}]
    By Lemma~\ref{lem:solidContain},
    we have
    \[
        \sup_{X\in \calX_P} \iprod{T, X}\le \max_{j\in \{0, \ldots, r\}} \max_{W\in \calW_j} \iprod{T, W}\,.
    \]
    Taking a union bound over all $j\in \{0, \ldots, r\}$ and using Lemmas~\ref{lem:0tail} and~\ref{lem:jTail}, we obtain, for every $t>0$,
    \[
        \Pr\left(\sup_{X\in \calX_P} \iprod{T, X}\gtrsim\frac {Lr^3} \eps(r+t)\right)\le (r+1)n^{-t}\,.
    \]
    Finally, using~\eqref{eq:rAssumption},
    the factor $r+1$ can be absorbed
    into the implicit constant after shifting $t$.
\end{proof}

%% file: ExtensionsVJ26.tex
\section{Extensions to arbitrary order and unbounded entries}
\label{sec:extensions}

In this section, we give extensions of Theorem~\ref{thm:mainTensorProof}.

First, if the entries of the random tensor take values in $[-M, M]$ rather than $[-1, 1]$, rescaling yields $\E \normin{T - \E T} \lesssim MLr^4/\eps$ under otherwise identical assumptions.
Moreover, by padding the dimensions, the theorem also applies to tensors $T\in \R^{n_1 \times \dots \times n_r}$, by setting $n = \max(n_1, \ldots, n_r)$.

Next, we combine Theorem~\ref{thm:mainTensorProof} with the following recent result of Boedihardjo~\cite{marchTensor} to
obtain a bound for arbitrary
$r$, without assuming~\eqref{eq:rAssumption},
at the cost of a slightly worse
dependence on $r$ and $\eps$.

\begin{theorem}[{\cite[Corollary 1.4]{marchTensor}}]\label{thm:march}
    Let $T\in (\R^n)^{\otimes r}$
    be a random tensor with 
    independent, mean-zero entries taking
    values in $[-K,K]$. Then,
    \[
        \E \normin{T}\lesssim \sqrt r \sum_{k=1}^r \max_{\substack{i_1, \ldots, i_{k-1}, i_{k+1}, \ldots, i_r\in [n]}} \left(\sum_{i_k=1}^n \E T^2_{i_1, \ldots, i_r}\right)^{\frac 12} + K r^3 \log^2 n\,.
    \]
    Moreover, for all $t\ge 1$,
    \[
        \Pr \left(\abs{\normin{T}- \E \normin{T}} \gtrsim Kt\right) \le e^{-t^2}\,.
    \]
\end{theorem}

\noindent In the setting of Theorem~\ref{thm:mainTensorProof}, Theorem~\ref{thm:march} does not give a
dimension-free bound. However, it applies to arbitrary $r$, independent of $n$ and $\eps$.

\begin{corollary}[Extension of Theorem~\ref{thm:mainTensorProof} to arbitrary $r$]\label{corollary:unconstrainedorder}
    Let $n,r\ge 2$ be integers,
    and let $\eps\in (0,1)$.
    Let $T\in (\R^n)^{\otimes r}$ be
    a random tensor with independent
    entries in $[-1,1]$
    and such that for some $L\ge 1$, we have $\Pr\left(T_i\neq 0\right)\le Ln^{-1-\eps}$ for all $i\in [n]^r$. Then,
    \[
        \E \normin{T - \E T}\lesssim \frac {Lr^5} {\eps^2}\log^2 \frac r \eps\,.
    \]
    Moreover, for any $t>0$,
    \[
        \Pr\left(\normin{T - \E T}\gtrsim \frac {Lr^3} {\eps}\left( \frac {r^2} {\eps}\log^2 \frac r \eps + t\right)\right)\le n^{-t}\,.
    \]
\end{corollary}

\begin{proof}
   Up to rescaling $L$, we assume that $n\ge 5$ without loss of generality.
   If $r$ satisfies~\eqref{eq:rAssumption},
    the result follows from Theorem~\ref{thm:mainTensorProof}.
    Suppose instead that $r>\frac {\eps \log n} {5\log \log n}$. By 
    assumption, for any  $k\in [r]$ and $i_1, \ldots, i_{k-1}, i_{k+1}, \ldots, i_r\in [n]$, we have
    \[
        \left(\sum_{i_k=1}^n \Var T_{i_1, \ldots, i_r}\right)^{\frac 12}\le \sqrt Ln^{-\frac \eps 2}\,.
    \]
    Applying Theorem~\ref{thm:march}
    gives
    \[
        \E \normin{T - \E T}\lesssim \sqrt {L} r^{\frac 32} n^{-\frac \eps 2} + r^3 \log^2 n\,.
    \]
    The first term is bounded by the target bound. For the second term, the assumption on $r$ gives $n\le (r/\eps)^{O(r/\eps)}$, so $r^3 \log^2 n\lesssim \frac {r^5} {\eps^2} \log^2 \frac r \eps$. For the tail bound, the expectation bound with an additive term $O(\sqrt{ t\log n}) \lesssim (t+1)r/\eps$ holds with probability at least $1-n^{-t}$.
\end{proof}

\noindent The almost sure
boundedness assumption on the entries
can also be removed via a truncation argument.

\begin{corollary}[Extension of Theorem~\ref{thm:mainTensorProof} to unbounded entries]\label{corollary:unboundedsparsetensor}
    Consider the setting of Corollary~\ref{corollary:unconstrainedorder}, except that the entries of $T$ may be unbounded.
    Let
    \begin{equation}
        C(r,\eps) \defeq \begin{cases}\frac {r^4} \eps & \text{ if \eqref{eq:rAssumption} is satisfied,}\\\frac{r^5}{ \eps^2} \log^2 \frac r \eps & \text{ otherwise.}\end{cases}\label{eq:constantCorollary}
    \end{equation}
    Then,
    \[
        \E \normin{T-\E T}\lesssim C(r,\eps) L \E \Normi{T}\,.
    \]
    Moreover, for any $t>0$ and $M \ge 2\E \Normi{T}$,
    \[
        \Pr \left( \normin{T-\E T} \gtrsim   LM\left(C(r,\eps) + \frac{r^{3}t}{\eps} \right)\right) \leq n^{- t} + \Pr\left( \Normi{T} > M\right)\,.
    \]
\end{corollary}

\begin{proof}
    Let $M\ge 2 \E \Normi{T}$ and decompose $T = T^{>M} + T^{\le M} $ where $T^{>M}$ contains the entries of magnitude greater than $M$ and $T^{\le M}$ contains the remaining entries.
    \[
        \normin{T- \E T} \leq  \normin{T^{>M} - \E T^{>M}} + \normin{T^{\le M} - \E T^{\le M}}\,.
    \]
    For the expectation bound, we apply Corollary~\ref{corollary:unconstrainedorder} (or Theorem~\ref{thm:mainTensorProof} if~\eqref{eq:rAssumption} is satisfied) to the second term.
    For the first term, we use the
    crude bound
    \[
        \E \normin{T^{>M} - \E T^{>M}}\lesssim  \E \normin{T^{>M}} \le \E \Norm{T^{>M}}_1\lesssim M\,,
    \]
    where the last inequality
    follows from Lemma~\ref{lemma:reversetrunc} (in the special case $M = 2 \E \Normi{T}$). The tail bound follows similarly,
    using $\Pr\left( T^{> M} \neq 0\right) = \Pr\left( \Normi{T} > M\right)$ and $\normin{ \E T^{> M}} \lesssim M $.
\end{proof}

\section{Tensors with i.i.d.\ entries}
\label{sec:seginer}

In this section, we extend Seginer's theorem (Theorem~\ref{thm:seginerMatrix}) to tensors under a $L^{2+\eps}$-$L^2$ moment equivalence assumption (Assumption~\ref{ass:momentGrowthAssumption}).  
We do not state a tail bound, although similar arguments to the ones in Section~\ref{sec:extensions} yield one. 

\begin{theorem}[Seginer-type theorem for tensors]\label{thm:seginer}
    Let $r\ge 2$ and $n_1, \ldots, n_r\ge 2$ be integers. Let $T \in \R^{n_1 \times \cdots \times n_r}$ be a tensor whose entries are i.i.d.\ copies of a random variable $Z$
    satisfying Assumption~\ref{ass:momentGrowthAssumption} with parameters $\varepsilon\in (0,1)$ and $K\ge 1$.
    Set $M \defeq 2\E \Normi{T}$ and $n \defeq \max_{i\in [r]} n_i$. Then,
    \[
        \sqrt{n \E\left[\clip_M(Z)^2\right]} + M \lesssim \E \normin{T} \lesssim r^3\sqrt{n \E\left[\clip_M(Z)^2\right]} + \frac{r^5 \log^2 \frac {r} \eps }{ \eps^{6} } KM \,.
    \]
\end{theorem}

Both the upper and lower bound in Theorem~\ref{thm:seginerIntro} can be deduced from Theorem~\ref{thm:seginer} (see the proof below). We note that
the dependence on $1/\eps$ of the lower bound
in Theorem~\ref{thm:seginerIntro} is necessarily
exponential, as the following example shows.

\begin{example}
    For any fixed $K\ge 1$ and $\eps\in (0,1)$, let 
    \[
        Z = \begin{cases}K^{\frac 1\eps} & \text{ with probability $\frac{K^{-\frac 2\eps}} 2$,}\\
        -K^{\frac 1 \eps} & \text{ with probability $\frac{K^{-\frac 2\eps}} 2$,}\\
        0 & \text{ with probability $1-{K^{-\frac 2\eps}}$.}
        \end{cases}
    \]
    Then $\E Z=0$, $\E Z^2=1$, and $\E |Z|^{2+\eps}=K$ so Assumption~\ref{ass:momentGrowthAssumption} is verified, but $\E \Abs{Z} = K^{-\frac 1 \eps}$. Thus, $\E \normin{T} \leq \E \Norm{T}_1 =   K^{-\frac 1 \eps} \prod_{i=1}^r n_i$ and hence for $K \gg (\prod_{i=1}^r n_i)^{\eps}$ one needs to divide by exponential factors in $1/\eps$ to make a lower bound in terms of the second moment valid.
\end{example}

\begin{proof}[Proofs of Theorems~\ref{thm:seginer}~and~\ref{thm:seginerIntro}]
    We start with the lower bound.
    Clearly, $\E \normin{T}\ge \E \Normi{T}= M/2$. To obtain the other term in the lower bound, we use the fact that the injective norm is at least the $\ell_2$-norm of any row. Without loss of generality, assume $n_r = n$. Then 
    \begin{align}
        \E \normin{T} &\geq \E \normin{\clip_M(T)} - \E \normin{T - \clip_M(T)}\nonumber\\
        &\geq \E \Norm{\clip_M(T)_{1, \ldots,1,\bullet}}_2 - \E \Norm{T - \clip_M(T)}_1\,.\label{eq:colminustrunc}
    \end{align}
    By Lemma~\ref{lemma:reversetrunc},
    the second term satisfies $\E \Norm{T - \clip_M(T)}_1\lesssim M$.
    For the first term, we apply Lemma~\ref{claim:boundNorm}. Since
    $\E \Norm{\clip_M(T)_{1, \ldots,1,\bullet}}^2_2 = n \E \left[\clip_M(Z)^2\right]$, we obtain
    \[
        \E \normin{T} \ge C \sqrt{ {n \E \left[\clip_M(Z)^2\right]}} - C' M\,,
    \]
    for some universal constants $C,C'>0$,
    which concludes the proof of the lower bound of Theorem~\ref{thm:seginer}. To deduce the lower bound of Theorem~\ref{thm:seginerIntro}, note that by Lemma~\ref{lemma:reversetrunc},
    \[
        \sqrt{n \E\left[\clip_M(Z)^2\right]} + M  \gtrsim \sqrt{n \E\left[Z^2\mathbf{1}_{\abs{Z} \leq M}\right]} + n \E \left[\abs{Z}\mathbf{1}_{\abs{Z} > M}\right] \gtrsim \sqrt{n}\E \abs{Z} \, ,
    \]  
    and the right-hand side is at least $\sqrt{n}\sigma K^{-\frac 1 \eps}$ by Lemma~\ref{lemma:paleyzygmund}.
    
    For the upper bound, set 
    \[
        B \defeq \min \left(\frac{\sqrt{n \E\left[\clip_M(Z)^2\right]} }{\log^2 n},  M \right)
    \]
    as the clipping threshold. We use the triangle inequality and bound separately the norms of the clipped and unclipped tensors. For the clipped part, Theorem~\ref{thm:march} implies 
    \[
        \E \normin{ \operatorname{clip}_B(T)- \E \left[\operatorname{clip}_B(T) \right]} \lesssim r^3 \left( \sqrt{n \E\left[\operatorname{clip}_B(Z)^2 \right]} + B \log^2 n \right) \lesssim  r^3 \sqrt{n \E\left[\clip_M(Z)^2 \right]}\,.
    \]
    For the unclipped part, we consider two cases:
    if $B=M$, then  Lemma~\ref{lemma:reversetrunc} directly gives
    \[
        \E \normin{ T- \operatorname{clip}_B(T) - \E \left[T- \operatorname{clip}_B(T)\right]} \le  \E \Norm{ T- \operatorname{clip}_B(T) - \E \left[T- \operatorname{clip}_B(T)\right]}_1 \lesssim M\,.
    \]
    If $B <M$, we apply Corollary~\ref{corollary:unboundedsparsetensor}. To apply it, we first verify its sparsity assumption. Note that
    \[
        \Pr\left(Z \neq \clip_B(Z) \right) = \Pr\left(|Z| >B \right) = \Pr\left(\abs{\clip_M(Z)} >B \right)  \leq \frac{\E \abs{\clip_M(Z)}^{2 + \eps}}{ B^{2+\eps}}\,.
    \]
    Here, the second equality uses $B<M$, and the inequality follows from Markov’s inequality. Using Lemma~\ref{lemma:equivalenceconstanttruncation} and Assumption~\ref{ass:momentGrowthAssumption}, we have
    \[
        \E \abs{\clip_M(Z)}^{2 + \eps}\le \left(\E {\clip_M(Z)}^2\right)^{1+\frac \eps 2} \frac {\E Z^{2+\eps}} {(\E Z^2)^{1+\frac \eps 2}}\le K \left(\E {\clip_M(Z)}^2\right)^{1+\frac \eps 2}\,.
    \]
    Combining these estimates with the definition of $B$ gives
    \[
        \Pr\left(Z - \clip_B(Z) \ne 0\right)\le \frac {K\log^{4+2\eps} n} {n^{1+\frac \eps 2}}\lesssim \frac {K} {n^{1+\frac \eps 4} \eps^4}\,.
    \]
    Thus, we can set $L \defeq C\frac{K }{ \eps^{4} }$ for some constant $C>0$ and use Corollary~\ref{corollary:unboundedsparsetensor} to obtain
    \[
        \E \normin{ T- \operatorname{clip}_B(T) - \E \left[T- \operatorname{clip}_B(T)\right]} \lesssim \frac{r^5 \log^2 \frac {4r} \eps }{ \eps^{6} } KM\,.
    \]
    This bounds the unclipped part in both cases and completes the proof.
\end{proof}

\begin{remark}[Necessity of the dependence on $\varepsilon$]\label{rem:tightnessSeginer}
    As in Section~\ref{sec:extensions}, the dependence on $\varepsilon$ and $r$ can be improved in certain parameter regimes; for example, when $r$ satisfies~\eqref{eq:rAssumption} or when $n^{\frac \eps 4} \geq \log^{4+2\eps} n$. However, the divergence as $\varepsilon\to 0$ is unavoidable.
    For example, let $A$ be an $n\times n$
    random matrix such that
    \[
        A_{ij}\overset{\text{i.i.d.}}\sim \frac 1{2n} (\delta_{\sqrt n} + \delta_{-\sqrt n}) + \frac 12 \left(1 - \frac 1{n}\right)(\delta_1 + \delta_{-1})\,.
    \]
    Then the maximum column norm
    of $A$ is of order $\sqrt{\frac {n\log n} {\log \log n}}$. On the other hand, if we denote by $\sigma$ the standard deviation of the entries, then $\sigma \sqrt n + \E \|A\|_\infty = \Theta(\sqrt n)$.
\end{remark}

%% file: AuxillaryLemmasVJ26.tex
\section{Auxiliary lemmas}\label{sec:truncationinequalities}

\begin{lemma}\label{lemma:subsetunionboundseries}
    Let $r\in \N$, $n \geq 2$ and $\lambda > 2$. Then,
    $$ \sum_{k_1, \ldots, k_r =1}^n \prod_{i=1}^r \binom n {k_i} n^{-\lambda k_i} \leq n^{-r(\lambda -2)}.$$
\end{lemma}

\begin{proof}
    Using $\binom{n}{k} \leq \frac{n^k}{k!}$, we get
    \[
        \sum_{k_1, \ldots, k_r =1}^n \binom{n}{k_1} \cdots \binom{n}{k_r} n^{-\lambda(k_1 + \ldots +k_r)} \leq \sum_{k_1, \ldots, k_r =1}^\infty \frac{ n^{-(\lambda-1)(k_1 + \ldots +k_r)}}{k_1! \cdots k_r!} = (\exp(n^{-(\lambda-1)})-1)^r\,.
    \]
    For $x\in [0,\frac 12]$, we have $e^x-1 \leq 2x$, and since $n \geq 2$, this gives
    \[
        (\exp(n^{-(\lambda-1)})-1)^r \leq \left( {2}n^{-(\lambda -1)}\right)^r \leq n^{-r(\lambda -2)}\,.\qedhere
    \]
\end{proof}

\begin{lemma}\label{lemma:reversetrunc}
    Let $X_1, \ldots, X_N\ge 0$ be independent random variables and $M \defeq 2\E \max_{i\in [N]} X_i$. Then,
    \[
        \sum_{i=1}^N \E \left[X_i \mathbf{1}_{X_i > M}\right] \lesssim M\,.
    \]
\end{lemma}

\begin{proof}
    We start with the tail integration formula
    \begin{align}
            \sum_{i=1}^N \E \left[X_i \mathbf{1}_{X_i > M}\right] &= \sum_{i=1}^N \int_{0}^{\infty} \Pr \left( X_i \mathbf{1}_{X_i > M}> t  \right) \, \mathrm dt\nonumber\\ 
            &= M \sum_{i=1}^N \Pr \left( X_i> M  \right) + \int_{M}^{\infty} \sum_{i=1}^N \Pr \left( X_i > t  \right) \, \mathrm dt\,.\label{eq:layertrunc}
    \end{align}
    By independence,
    \[
        \sum_{i=1}^N \Pr \left( X_i > t  \right) \leq - \log \Pr \left( \max_{i\in [N]} X_i \leq t  \right) = - \log\left( 1- \Pr\left( \max_{i\in [N]} X_i > t  \right)\right)\,.
    \]
    For $t \geq M$, Markov's inequality gives $\Pr \left( \max_{i\in [N]} X_i > t \right) \leq \frac 12$. We use the fact that $-\log(1-x) \leq 2x$ holds for $0 \leq x \leq \frac 12$ and obtain
    \[
        - \log\left( 1- \Pr \left( \max_{i\in [N]} X_i > t  \right)\right) \leq 2\Pr \left( \max_{i\in [N]} X_i > t  \right)\,.
    \]
    Apply this estimate to both terms in~\eqref{eq:layertrunc}. For the first term, Markov's inequality gives a bound of $M$. For the integral term, the tail integration formula gives the same bound.
\end{proof}

\begin{lemma}\label{claim:boundNorm}
    Let $X$ be a random vector with independent entries
    satisfying $\|X\|_\infty\le M$ 
     almost surely for some $M>0$. Then,
     \[\left(\E \Norm{X}^2_2\right)^{1/2}\le \sqrt 2 \E \Norm{X}_2 + M\,.\]
\end{lemma}

\begin{proof}
    Applying $\sqrt{x} - \sqrt{y}\le \sqrt{|x-y|}$ and then Jensen's inequality, we have
    \[
        \left(\E \Norm{X}_2^2\right)^{\frac 12}\le \E \Norm{X}_2 + \E \abs{\Norm{X}_2^2 - \E \Norm{X}_2^2}^{\frac 12}\le \E \Norm{X}_2 + \left(\Var \Norm{X}_2^2\right)^{\frac 14}\,.
    \]
    By independence and boundedness of the entries,
    \[
        \Var \Norm{X}_2^2 = \sum_i \Var X_i^2\le \sum_i \E X_i^4\le M^2 \sum_i \E X_i^2 = M^2 \E \Norm{X}_2^2\,.
    \]
    Using $\sqrt{xy}\le \frac {x+y} 2$
    and combining these estimates,
    \[
         \sqrt{\E \Norm{X}_2^2}\le \E \Norm{X}_2 + \frac {M + \sqrt{\E \Norm{X}_2^2}} 2\,,
    \]
    and the claim follows by rearranging the inequality.
\end{proof}

\begin{lemma}\label{lemma:equivalenceconstanttruncation}
    Let $0<q<p$, and let $X$ be a random variable such
    that $\E \abs{X}^q>0$ and $\E \abs{X}^p<\infty$.
    Then for any $M>0$,
    \[
        \frac {(\E \abs{\clip_M(X)}^p)^{1/p}} {(\E \abs{\clip_M(X)}^q)^{1/q}} \le \frac {(\E \abs{X}^p)^{1/p}} {(\E \abs{X}^q)^{1/q}}\,.
    \]
\end{lemma}

\begin{proof}
    Let
    \[
        m_{p}(M) \coloneqq \E \abs{\clip_M(X)}^p = p \int_0^M  t^{p-1} \Pr\left( |X|> t\right) \diff t\,,
    \]
    where the equality follows from the tail integration formula and a change of variables. Differentiating gives
    \[
        m_{p}'(M)  = p M^{p -1} \Pr\left(|X| > M\right)\,.
    \]
    The logarithmic derivative of the quotient is therefore
    \[
        \left(\log \frac {m_{p}^{1/p}} {m_q^{1/q}}\right)'(M) = \frac 1p \frac {m_p'(M)} {m_p(M)} - \frac 1q \frac {m_q'(M)} {m_q(M)} = \Pr\left(\abs{X} > M\right)\left(\frac {M^{p-1}} {m_p(M)} - \frac {M^{q-1}} {m_q(M)}\right)\ge 0\,,
    \]
    since $m_p(M)\le M^{p-q} m_q(M)$.
    Hence, the function $M\mapsto \frac {m_p^{1/p}(M)}{m_q^{1/q}(M)}$ is non-decreasing.
    Taking the limit $M \to \infty$ proves the desired inequality.
\end{proof}

\begin{lemma}\label{lemma:paleyzygmund}
    Let $Z$ be a random variable satisfying
    Assumption~\ref{ass:momentGrowthAssumption}
    for some $\eps\in (0,1)$ and $K\ge 1$. Then,
    \[
        \frac {\sqrt{\E Z^2}} {\E \abs{Z}} \le K^{\frac 1 \eps}\,.
    \]
\end{lemma}

\begin{proof}
    By H{\"o}lder's inequality,
    \[
        \E Z^2\le (\E |Z|)^{\frac \eps {1+\eps}} (\E |Z|^{2+\eps})^{\frac 1 {1+\eps}}\le K^{\frac 1 {1+\eps}} (\E |Z|)^{\frac \eps {1+\eps}} (\E Z^2)^{\frac {1 + \frac \eps 2} {1+\eps}}\,.
    \]
    Rearranging the inequality gives the claim.
\end{proof}

\section{Seginer-type theorems under moment equivalence}
\label{sec:equivalenceSeginer}
\begin{proposition}[Equivalence between Theorem~\ref{thm:seginerIntro} and Conjecture~\ref{conj:seginerTensor} under Assumption~\ref{ass:momentGrowthAssumption}] \label{prop:proofSeginerEquivalence}
    Let $r\ge 2$ and $n_1, \ldots, n_r\ge 2$ be integers.
    Let $T\in \R^{n_1\times \ldots n_r}$ be a random tensor whose entries
    are i.i.d.\ copies of a random variable satisfying Assumption~\ref{ass:momentGrowthAssumption}
    for some $\eps\in (0,1)$ and $K\ge 1$,
    and let $\sigma$ denote the standard deviation of that random variable. Then, letting $n\defeq \max(n_1, \ldots, n_r)$,
    \[
        \max_{k\in [r]} \E \left[\max_{\substack{i_1\in [n_1], \ldots, i_{k-1}\in [n_{k-1}], i_{k+1}\in [n_{k+1}], \ldots, i_r\in [n_r]}} \left(\sum_{i_k=1}^{n_k} T_{i_1, \ldots, i_r}^2\right)^{1/2}\right] \underset{r,K,\eps}\asymp \sigma \sqrt n + \E \Normi{T}\,.
    \]
\end{proposition}

\begin{proof}
    To simplify notations, we assume without loss of generality that $n_r=n$.
    For the lower bound, it suffices to lower bound the term for $k=r$. Note first that
    the inequality
    \begin{equation}
        \max_{i_1,\ldots,i_{r-1}} \left(\sum_{i_r=1}^n T_{i_1, \ldots, i_r}^2\right)^{1/2}\ge \Normi{T}\label{eq:lb1}
    \end{equation}
    holds pointwise. For the other term, using
    $\|x\|_1\le \sqrt n\cdot \|x\|_2$ and Lemma~\ref{lemma:paleyzygmund},
    \begin{equation}
        \E \max_{i_1, \ldots, i_{r-1}} \Norm{T_{i_1, \ldots, i_{r-1},\bullet}}_2\ge \E \Norm{T_{1,\ldots,1,\bullet}}_2\ge \frac {\E \Norm{T_{1,\ldots,1,\bullet}}_1} {\sqrt n} \ge K^{-\frac 1 \eps} \sigma \sqrt{n}\,.\label{eq:lb2}
    \end{equation}
    Combining \eqref{eq:lb1} and~\eqref{eq:lb2} completes the proof of the lower bound.
    
    For the upper bound, fix a threshold $B>0$ and decompose
    $T = T^{\le B} + T^{>B}$, where $T^{\le B}$ contains the entries of $T$
    of magnitude at most $B$, and $T^{>B}$ contains the remaining entries. Similarly, we use the notation
    $Z^{\le B}\defeq Z\mathbf 1_{|Z|\le B}$ and $Z^{>B}\defeq Z \mathbf 1_{|Z|>B}$, where $Z$ is a random variable distributed as the entries of $T$.
    
    To simplify notation, we treat the case $k=r$; the same argument applies to arbitrary $k\in [r]$.
    Moreover, by the triangle inequality, it suffices to bound $\E \max_{i_1, \ldots, i_{r-1}} \|T^{\le B}_{i_1, \ldots, i_{r-1},\bullet}\|_2$ and $\E \max_{i_1, \ldots, i_{r-1}} \|T^{>B}_{i_1, \ldots, i_{r-1},\bullet}\|_2$.

    We start with $T^{\le B}$. Bernstein's inequality (Lemma~\ref{lem:bernstein}), followed by a union bound over the rows $(i_1, \ldots, i_{r-1})$ and integrating the resulting tail bound, implies
    \[
        \E \max_{i_1, \ldots, i_{r-1}} \|T^{\le B}_{i_1, \ldots, i_{r-1},\bullet}\|_2^2\lesssim n \E \left[\left(Z^{\le B}\right)^2\right] +  \left(nr \log n \Var \left[\left(Z^{\le B}\right)^2\right]\right)^{1/2} + B^2 r \log n\,.
    \]
    Since
    \[
        \E \left[\left(Z^{\le B}\right)^2\right]\le \sigma^2\quad \text{and} \quad\Var \left[\left(Z^{\le B}\right)^2\right]\le B^2 \sigma^2\,,
    \]
    we choose $B \defeq \sigma \sqrt{n/\log n}$. Then, by Jensen's inequality, we obtain 
    \begin{equation}
        \E \max_{i_1, \ldots, i_{r-1}} \|T^{\le B}_{i_1, \ldots, i_{r-1},\bullet}\|_2\le \left(\E \max_{i_1, \ldots, i_{r-1}} \|T^{\le B}_{i_1, \ldots, i_{r-1},\bullet}\|^2_2\right)^{\frac 12} \lesssim_r  \sigma \sqrt {n}\,, \label{eq:ub1}
    \end{equation}
    as desired. 
    It remains to handle $T^{>B}$. By
    Markov's inequality and Assumption~\ref{ass:momentGrowthAssumption},
    \[
        \Pr\left( Z^{>B}\neq 0\right) \le \frac {K \sigma^{2+\eps}} {B^{2+\eps}} \le \frac {K \log^{1+\frac \eps 2} n} {n^{1+\frac \eps 2}}\,.
    \]
    Applying the Chernoff bound (Lemma~\ref{lemma:chernoff}) to bound
    the number of nonzero entries in one
    row, followed by a union bound over all
    rows, yields the following: there exists $C=C(\eps,K,r)\ge 1$ such that
    $T^{>B}$ has at most $C$ nonzero entries in each
    row $(i_1, \ldots, i_{r-1})$, with probability at least $1-n^{-r}$. Let $\Omega$
    denote this event. Then,
    \begin{align}
        \E \max_{i_1, \ldots, i_{r-1}} \norm{T_{i_1, \ldots, i_{r-1},\bullet}^{>B}}_2 &= \E \left[\max_{i_1, \ldots, i_{r-1}} \norm{T_{i_1, \ldots, i_{r-1},\bullet}^{>B}}_2 \mathbf 1_{\Omega}\right] + \E \left[\max_{i_1, \ldots, i_{r-1}} \norm{T_{i_1, \ldots, i_{r-1},\bullet}^{>B}}_2 \mathbf 1_{\Omega^{\mathrm c}}\right]\nonumber\\
        &\le \sqrt C \E \Normi{T} + \Pr\left(\Omega^{\mathrm c}\right)^{\frac 12} \left(\E \Norm{T}_2^2\right)^{\frac 12}\nonumber\\
        &\lesssim_{\eps,K,r} \E \Norm{T}_\infty + \sigma \sqrt n\,.\label{eq:ub2}
    \end{align}
    Combining~\eqref{eq:ub1} and~\eqref{eq:ub2} completes the proof.
\end{proof}